\documentclass{amsart}
\usepackage{amsmath,amssymb,amsthm,hyperref,amsrefs}
\usepackage{amsfonts,amsmath,mathrsfs,amssymb,graphicx,multirow,amsthm,latexsym,xcolor}
\usepackage[mathscr]{euscript}
\usepackage[all]{xy}
\usepackage{enumerate}
\usepackage{multicol}

\allowdisplaybreaks

\newcommand{\vv}[1]{\underline{\vphantom{y}{#1}}}
\newcommand{\Cone}[1]{\operatorname{Cone}\!\left({#1}\right)}
\newcommand{\Hm}[2]{\operatorname{H}_{#1}\!\left({#2}\right)}
\newcommand{\betti}[3][R]{\beta^{#1}_{#2}\!\left({#3}\!\right)}

\newcommand{\mcI}{\mathcal{I}}
\newcommand{\mcJ}{\mathcal{J}}

\newcommand{\Tor}[4][R]{\operatorname{Tor}^{#1}_{#2}\!\left(#3,#4\right)}

\newcommand{\ldb}{\mathopen{[\![}} 
\newcommand{\rdb}{\mathclose{]\!]}} 
\newcommand{\lp}{\left(}
\newcommand{\rp}{\right)}
\newcommand{\m}{\mathfrak{m}}

\newcommand{\wt}{\widetilde}

\theoremstyle{plain}
\newtheorem{theorem}{Theorem}
\newtheorem{prop}[theorem]{Proposition}
\newtheorem{cor}[theorem]{Corollary}
\newtheorem{lemma}[theorem]{Lemma}

\newtheorem{maintheorem}{Theorem}

\theoremstyle{definition}
\newtheorem{definition}[theorem]{Definition}
\newtheorem{notation}[theorem]{Notation}

\newtheorem{example}[theorem]{Example}
\newtheorem{rem}[theorem]{Remark}

\numberwithin{theorem}{section}
\numberwithin{equation}{section}
\numberwithin{equation}{theorem}

\begin{document}

\title{Classifying Betti Numbers of Fiber Products}
\author[Ela Celikbas]{Ela Celikbas}
\address{408K Armstrong Hall, School of Mathematical and Data Sciences, West Virginia University, Morgantown, WV 26506, USA.}
\email{ela.celikbas@math.wvu.edu}

\author[Hugh Geller]{Hugh Geller}
\address{411A Armstrong Hall, School of Mathematical and Data Sciences, West Virginia University, Morgantown, WV 26506, USA.}
\email{hugh.geller@mail.wvu.edu}

\author[Tony Se]{Tony Se}
\address{203 Jackson-Davis Hall, Department of Mathematics,
Florida A\&M University, Tallahassee, FL 32307, USA.}
\email{tony.se@famu.edu}

\keywords{Fiber products, Free resolution, Betti numbers, Poincar\'e series}
\subjclass[2020]{13D02, 13D07}

\begin{abstract}
We consider fiber products of complete, local, noetherian algebras over a fixed residue field. Some of these rings cannot be minimally resolved with a free resolution using the recent work of the second author. We develop techniques to modify Geller's approach in order to recover the Betti numbers and Poincar\'e series for these fiber products.

\end{abstract}
\date{\today}
\maketitle

\section{Introduction}
We consider two commutative, local, noetherian rings $(S, \m_S,k)$ and $(T, \m_T,k)$. The {\it fiber product} (or {\it pullback}) $S \times_k T$  of these rings over $k$ is defined as the set $\{(s,t) \in S \times T: \pi_S(s) = \pi_T(t)\}$, where $S \xrightarrow{\pi_S} k \xleftarrow{\pi_T} T$ are the natural surjections. It is well-known that $S \times_k T$ is a local ring with maximal ideal $\m_{S\times_k T}=\m_S\oplus \m_T$ and residue field $k$. 
Fiber products find frequent applications in commutative algebra, category theory, and algebraic geometry. Since the 1980s, extensive research has focused on the ring theoretic and homological properties of fiber products, see  \cites{MR0682707, MR0771818, MR0951203, MR2929675, MR3924433, MR4078358, MR4273207, MR0647683, MR2488551, CCCEGIM2023Arf, MR3988200, MR4232178, MR3691985, MR3862678, MR4064107, MR2580452}.

Numerous findings indicate that the properties of $S \times_k T$ resemble those of the rings $S$ and $T$ \cites{MR2929675,CCCEGIM2023Arf,MR4273207,MR3862678}. 
Additionally, the module category of $S \times_k T$ exhibits a significant relationship with those of $S$ and $T$ \cites{MR0092776,MR0647683}. In \cite{MR3691985}, Nasseh and Sather-Wagstaff explored modules over $S \times_k T$ and presented counterarguments challenging these notions. They examined some consequences for depth formulas (see also \cite{MR2580452}*{Remark 3.1}), rigidity properties of $\text{Tor}$, and various $\text{Ext}$ vanishing results over $S \times_k T$. As a consequence, they also established that the highly coveted Auslander-Reiten conjecture holds for the fiber product $S \times_k T$ if $S \neq k \neq T$.

In 2009, Moore constructed a minimal free resolution over $S \times_k T$ for an $S$-module $M$ using minimal resolutions of $M$ and $k$ over $S$, as well as one of $k$ over $T$. This resolution provided precise information on the multiplicative structure of cohomology over $S \times_k T$ and computed depths of cohomology modules over the fiber product, see \cite{MR2488551}. Schnibben later worked on Golod homomorphisms between specific fiber products and provided a module resolution construction, see \cite{SchnibbenThesis}.

Recently, in \cite{MR4470165}, the second author of this paper gives an explicit construction of free resolutions for fiber products of the form $\frac{R}{\mcI' + \mcJ} \times_{\frac{R}{\mcI + \mcJ}} \frac{R}{\mcI + \mcJ'}$ over the ring $R$, where the ideals (of $R$) $\mcI' \subseteq \mcI$ and $\mcJ' \subseteq \mcJ$ satisfy certain $\operatorname{Tor}$-vanishing properties. Geller proves the construction is minimal if $\Tor{i}{\phi}{k} = 0 = \Tor{i}{\psi}{k}$ for all $i \geq 1$ where $\phi$ and $\psi$ are liftings of the natural surjections $\frac{R}{\mcI'} \to \frac{R}{\mcI}$ and $\frac{R}{\mcJ'} \to \frac{R}{\mcJ}$, respectively. In this paper, we address cases of fiber products that fail to meet the minimality criterion or do not match the form covered in \cite{MR4470165}.

In our work, we set $\vv{x} = x_1, \ldots, x_n$ and $\vv{y} = y_1, \ldots, y_{n'}$ and consider fiber products $F = (k \ldb \vv{x} \rdb / I) \times_k (k \ldb \vv{y} \rdb / I')$ where at least one of the containments of ideals $I \subseteq ( \vv{x} )^2$ and $I' \subseteq ( \vv{y} )^2$ does not hold. In this case, the resolutions from \cite{MR4470165} fail to be minimal, so we reduce $F$ to a new fiber product $\tilde{F}$. We apply the results of \cite{MR4470165} to $\widetilde{F}$ and then use combinatorial techniques to obtain the first of our two main theorems.

\begin{maintheorem}[Theorem \ref{thm:Betti}]\label{maintheoremA}
Take $I$, $I'$, and $F$ as above and set $R = k \ldb \vv{x}, \vv{y} \rdb$. If $p~=~\dim_k\left(\frac{I + \lp \vv{x} \rp^2}{\lp \vv{x} \rp^2}\right)$ and $q~=~\dim_k\left(\frac{I' + \lp \vv{y} \rp^2}{\lp \vv{y} \rp^2}\right)$, then the Betti numbers of $F$ over $R$ can be written in terms of $p$, $q$, $n$, $n'$, the Betti numbers of $R / I$, and the Betti numbers of $R / I'$.
\end{maintheorem}

Our other main theorem uses similar strategies to establish a functional equation for the Poincar\'e series of $F$ over $R$.

\begin{maintheorem}[Theorem \ref{thm:Poin}]\label{maintheoremB}
Assume the set-up of Theorem \ref{maintheoremA}. The Poincar\'e series of $F$ over $R$ satisfies \[\frac{ P_F^R(t) - \lp 1 + t \rp^{n'} P_{R/I}^R(t) - \lp 1 + t \rp^{n} P_{R/I'}^R(t) + \lp 1 + t \rp^{n + n'} }{\lp \lp 1 + t \rp^{n \vphantom{n'}} - \lp 1 + t \rp^p \rp \lp \lp1 + t \rp^{n'} - \lp 1 + t \rp^q \rp} = \frac{t + 1}{t}.\]
\end{maintheorem}

In Section \ref{background}, we introduce the construction from \cite{MR4470165} and provide an example demonstrating why it is not minimal for the fiber product $F$. This motivates the definition of $p$ and $q$ as given in Theorem \ref{maintheoremA} that measures the degree to which the containment $I \subseteq ( \vv{x} )^2$ (and $I' \subseteq ( \vv{y} )^2$) fails. In Section~\ref{Results}, we use $p$ and $q$ to reduce $F$ to another fiber product $\widetilde{F}$. We then apply combinatorial techniques to results on $\widetilde{F}$ to obtain Theorems \ref{maintheoremA} and \ref{maintheoremB}.

\section{Background and Preliminary Results}\label{background}

Throughout the paper, let $k$ be a field and $R = k \ldb \vv{x}, \vv{y} \rdb$, where $\vv{x} = x_1, \ldots, x_n$ and $\vv{y} = y_1, \ldots, y_{n'}$. We consider two ideals $I \subseteq k \ldb \vv{x} \rdb$ and $I' \subseteq k \ldb \vv{y} \rdb$ and view them in $R$. To shorten notation, we write $I$ and $I'$ in place of $IR$ and $I'R$.

\begin{rem}[\cite{MR4470165}*{Proposition 2.1}]  \label{prop:defideal}
The defining ideal for the fiber product $\lp k \ldb \vv{x} \rdb / I \rp \times_k \lp k \ldb \vv{y} \rdb /I' \rp$ is given by $I + I' + (\vv{x}\vv{y})$ where $(\vv{x}\vv{y})$ is the ideal generated by the set $\{x_i y_j : 1 \leq i \leq n, 1 \leq j \leq n' \}$.
\end{rem}

\begin{example} \label{example:intro}
We consider the two ideals $I=(x_1 + x_1^3 + x_2^2) \subset k \ldb x_1,x_2 \rdb$ and $I'=(y_1 + y_1^3 + y_2^2) \subset k \ldb y_1, y_2 \rdb$. Then $k\ldb \vv{x} \rdb /I  \times_k k \ldb \vv{y} \rdb/ I'$ has a defining ideal $$(x_1 + x_1^3 + x_2^2, y_1 + y_1^3 + y_2^2, x_1y_1, x_1y_2, x_2y_1, x_2y_2).$$ However, we see that
\begin{align*}
\lp 1 + x_1^2 \rp x_1 y_1 &= y_1 \lp x_1 + x_1^3 + x_2^2 \rp - x_2 \lp x_2 y_1 \rp.
\end{align*}
Since $1 + x_1^2 \in k \ldb \vv{x}, \vv{y} \rdb^{\times}$, the above shows that $x_1 y_1$ can be written in terms of the other generators. That is, the list coming from Remark~\ref{prop:defideal} is not necessarily minimal. Similar computations hold for $x_1 y_2$ and $x_2 y_1$.
\end{example}

We now review the construction from \cite{MR4470165} that we will use in this paper. 

\begin{definition}
Let $\mathcal X$ and $\mathcal Y$ be complexes of free $R$-modules. The complex $\mathcal X \ast_R \mathcal Y$, previously called the \emph{star product} in \cites{MR4420538, MR4470165}, is defined by
$$(\mathcal X\ast_R \mathcal Y)_n=\begin{cases} (\mathcal X_{\geq 1} \otimes_R \mathcal Y_{\geq 1})_{n+1} & n\geq 1 \\ 
\mathcal X_{0}\otimes_R \mathcal Y_{0} & n=0 \\
0 & n<0\end{cases} \; \text{and}\; \partial_n^{\mathcal X\ast \mathcal Y}=\begin{cases} \partial_{n+1}^{\mathcal X_{\geq 1} \otimes_R \mathcal Y_{\geq 1}} & n\geq 2 \\ 
\partial_1^{\mathcal X} \otimes \partial_1^{\mathcal Y} & n=1 \\
0 & n\leq 0\end{cases}.$$
In other words, we get $\mathcal X \ast_R \mathcal Y$ by truncating $\mathcal X$ and $\mathcal Y$, tensoring the truncations, and then shifting and augmenting the tensor product.
\end{definition}

From now on, let $\mathcal X$, $\mathcal Y$, $\mathcal S$, $\mathcal T$ resolve the quotients $R/(\vv{x})$, $R/(\vv{y})$, $R/I$, and $R/I'$, respectively.
Let $\phi \colon \mathcal S \rightarrow \mathcal X$ and $\psi\colon \mathcal S\rightarrow \mathcal X$ be two chain maps lifting the natural surjections $R/I \rightarrow R/(\vv{x})$ and $R/I' \rightarrow R/(\vv{y})$, respectively. The maps $\phi$ and $\psi$ are then lifted to $\Phi:\Sigma^{-1} (\mathcal S_{\geq 1} \otimes_R \mathcal Y) \rightarrow \mathcal X\ast_R \mathcal Y$ and $\Psi: \Sigma^{-1} (\mathcal X \otimes_R \mathcal T_{\geq 1}) \rightarrow \mathcal X\ast_R \mathcal Y$ using \cite{MR4470165}*{Construction 4.5} and shown to be chain maps \cite{MR4470165}*{Lemma 4.6}. These maps are used in the following lemma.

\begin{lemma}\label{lem:RedCase}  Let $I \subseteq ( \vv{x} )^2 \subset k \ldb \vv{x} \rdb$ and $I' \subseteq ( \vv{y})^2 \subset k \ldb \vv{y} \rdb$. Let 
$$\Omega=\begin{pmatrix} \Phi & \Psi \end{pmatrix} \colon \Sigma^{-1} (\mathcal S_{\geq 1} \otimes_R \mathcal Y) \oplus \Sigma^{-1} (\mathcal X \otimes_R \mathcal T_{\geq 1}) \rightarrow \mathcal X \ast_R \mathcal Y.$$
Then $\Cone \Omega$ is a minimal free resolution of $\frac{k \ldb \vv{x} \rdb}{I}~\times_k~\frac{k \ldb \vv{y} \rdb}{I'}$ over~$R =  k \ldb \vv{x}, \vv{y} \rdb$.
\end{lemma}

\begin{proof}
By \cite{MR4470165}*{Example 4.2}, the pairs $\{I,(\vv{y})\}$, $\{( \vv{x} ), (\vv{y})\}$, and $\{(\vv{x}), I' \}$ satisfy the necessary $\operatorname{Tor}$-vanishing for $\Cone \Omega$ to produce a free resolution of the fiber product $\frac{k \ldb \vv{x} \rdb}{I}\!\times_k\!\frac{k \ldb \vv{y} \rdb}{I'}$ over $R$.

Observe that $\Tor{i}{\phi}{k} : \Tor{i}{\tfrac{R}{I}}{k} \to \Tor{i}{\tfrac{R}{( \vv{x} )}}{k}$ factors through $$\Tor{i}{\tfrac{R}{(\vv{x})^2}}{k} \to \Tor{i}{\frac{R}{(\vv{x})}}{k}$$ for $i \geq 0$. By \cite{MR3754407}*{Remark 2.12}, we have \[\Tor{i}{\tfrac{R}{(\vv{x})^2}}{k} \cong \Tor{i - 1}{(\vv{x} )^2}{k} \to \Tor{i - 1}{(\vv{x} )}{k} \cong \Tor{i}{\tfrac{R}{(\vv{x})}}{k}\] is the zero map for all $i \geq 1$. Thus, we conclude $\Tor{i}{\phi}{k} = 0$ is the zero map for all $i \geq 1$. The same argument holds for $\Tor{i}{\frac{R}{I'}}{k} \to \Tor{i}{\frac{R}{( \vv{y} )}}{k}$. Therefore the minimality conditions of \cite{MR4470165}*{Theorem 5.2} are satisfied.
\end{proof}

Lemma \ref{lem:RedCase} gives a sufficient condition (i.e. $I \subseteq (\vv{x})^2$ and $I' \subseteq (\vv{y})^2$) for the construction in \cite{MR4470165} to be minimal. The following example begins to address the case when we fail to meet the conditions of Lemma \ref{lem:RedCase}.

\begin{example}\label{example:original}
Consider $(x_1 + x_1^3 + x_2^2) \subset k \ldb x_1,x_2 \rdb$ and $(y_1 + y_1^3 + y_2^2) \subset k \ldb y_1, y_2 \rdb$. The resolution from \cite{MR4470165}*{Theorem 4.8} has the form
\[
\xymatrix{
0 \ar[r] & R^3 \ar[r] & R^8 \ar[r] & R^6 \ar[r] & R \ar[r] & 0
}
\]
but a computation in Macaulay2 gives Betti numbers $(1, 3, 3, 1)$. To see where we lose minimality, we consider the chain map obtained by lifting the natural surjection $R/(x_1 + x_1^3 + x_2^2) \to R/(x_1,x_2)$.
\[
\xymatrix{
\mathcal{S} := \ar[d]_-{\phi} & & 0 \ar[r] & R \ar[rr]^-{x_1 + x_1^3 + x_2^2} \ar[d]^-{\left(\!\begin{smallmatrix} x_1^2 + 1 \\ x_2 \end{smallmatrix}\!\right)} && R \ar[r] \ar[d]^-{1} & 0 \\
K^R(x_1,x_2) := & 0 \ar[r] & R \ar[r]_-{\left(\!\begin{smallmatrix} -x_2 \\ x_1 \end{smallmatrix}\!\right)} & R^2 \ar[rr]_-{\left(\!\begin{smallmatrix} x_1 & x_2 \end{smallmatrix}\!\right)} && R \ar[r] & 0
}
\]
The presence of a unit in the entries of $\phi_1$ yields $\Tor{1}{\phi}{k} \neq 0$ and thus the construction fails the minimality conditions of \cite{MR4470165}*{Theorem 5.2}.
\end{example}

\begin{lemma}\label{lem:tonylem}
Let $I=(g_1,g_2,\dots,g_m)\subseteq (\vv{x})$ and $p~=~\dim_k\left(\frac{I + \lp \vv{x} \rp^2}{\lp \vv{x} \rp^2}\right)$. Then $p \leq \min\{m,n\}$ and we can rewrite the $g_i$ such that:
\begin{enumerate}[(i)]
  \item the coefficient of $x_i$ is 1 in $g_i$ and 0 in $g_j$
  for all $i,j=1,2,\dots,p$ with $i \neq j$, and \label{item:kronecker}
  \item $g_{p+1},g_{p+2},\dots,g_m \in (\vv{x})^2$. \label{item:msquared}
\end{enumerate}
Moreover, $I = \lp \vv{x} \rp$ if and only if $p = n$.
\end{lemma}

\begin{proof}
For each $i=1,2,\dots,m$, let $g_i=a_{i1}x_1+a_{i2}x_2+\dots+a_{in}x_n+h_i$,
where each $a_{ij} \in k$ and $h_i \in (\vv{x})^2$. Performing Gaussian elimination on the $m$ equations produces a new set of generators for $I$ given by \[\wt{g_i} = \begin{cases} x_{j_i} + \wt{h_i} + \sum_{j > j_i} \wt{a_{ij}} x_j  & 1 \leq i \leq q \\ \wt{h}_i & q < i \leq m \end{cases}\] where each $\wt{a_{ij}} \in k$ and $\wt{h_i} \in \lp \vv{x} \rp^2$. Consequently, $\wt{g_i} \in \lp \vv{x} \rp^2$ for $q < i \leq m$, so \[\frac{I + \lp \vv{x} \rp^2}{\lp \vv{x} \rp^2} = \frac{\lp \wt{g_1}, \ldots, \wt{g_q}, \wt{g}_{q + 1}, \ldots, \wt{g_m} \rp + \lp \vv{x} \rp^2}{\lp \vv{x} \rp^2} = \frac{\lp \wt{g_1}, \ldots, \wt{g_q} \rp + \lp \vv{x} \rp^2}{\lp \vv{x} \rp^2}.\] The dimension over $k$ on the left-hand side is $p$ while the right-hand side is dimension $q$, so $p = q$. By re-indexing the variables so that $x_1,\dots,x_p$ are those with a leading 1, we obtain \eqref{item:kronecker} and \eqref{item:msquared} with $p \leq \min\{m,n\}$.

Finally, if $I = \lp \vv{x} \rp$, then we immediately have $p=n$. On the other hand, if $p = n$, then $\lp \vv{x} \rp =I+\lp \vv{x} \rp^2$. By Nakayama's Lemma, $I=\lp \vv{x}\rp$.
\end{proof}

We now use Lemma \ref{lem:tonylem} to demonstrate stricter conditions that may be imposed on the choices of $g_i$. In particular, we use the next few results to show that without loss of generality, we may assume $g_i = x_i$ for $1 \leq i \leq p$.

\begin{lemma} \label{lem:Tony}
  Let $I$ and $p$ be as in Lemma~\ref{lem:tonylem}, with $I$ satisfying \eqref{item:kronecker} and \eqref{item:msquared}. Then $g_1,g_2,\dots,g_p,x_{p+1}, \ldots, x_n$ is a regular sequence over $k \ldb \vv{x} \rdb$.
  
\end{lemma}

\begin{proof}
  Let $L=(g_1,g_2,\dots,g_p,x_{p+1},x_{p+2},\dots,x_n)$. For each $i=1,2,\dots,p$, let
  $\wt{g_i}=g_i(x_1,x_2,\dots,x_p,0,\dots,0)$.
  Then $L=(\wt{g_1},\wt{g_2},\dots,\wt{g_p},x_{p+1},x_{p+2},\dots,x_n)$,
  and $\lp \vv{x} \rp = L+\lp \vv{x} \rp^2$. By Nakayama's Lemma, $L=\lp \vv{x} \rp$.
  By \cite{MR1011461}*{Theorem~17.4}, $g_1,\dots,g_p, x_{p+1}, \ldots, x_n$ is a regular sequence over $k \ldb \vv{x} \rdb$.
\end{proof}

\begin{rem}\label{remark:xnotg}
Lemma~\ref{lem:Tony}, we have 
$g_1, \ldots, g_p, x_{p+1}, \ldots, x_n$ is a regular sequence over $k \ldb \vv{x} \rdb$.
 This allows us to use a change of variables to replace $x_i$ with $g_i$ for $1 \leq i \leq p$ by \cite{MR1011461}*{p.~223}. Hence, we may assume that $g_i=x_i$ for all $1\leq i \leq p$  
and $I$ is of the form $I = (x_1, \ldots, x_p, g_{p+1}, \ldots, g_m)$ where 
$x_1,\ldots,x_p,g_{p+1},\ldots,g_m$ is a minimal set of generators.
\end{rem}

The following lemma is a consequence of Remark \ref{remark:xnotg}.

\begin{lemma}\label{lem:torvanVA}
Take $I = (x_1,\ldots,x_p,g_{p+1},\ldots,g_m)$ as in Remark \ref{remark:xnotg}. One can construct an ideal $J \subseteq (x_{p+1},\ldots,x_n)^2$ such that $I = (x_1,\ldots,x_p) + J$ and $$\Tor[k\ldb \vv{x} \rdb]{i}{\frac{k\ldb \vv{x} \rdb}{(x_1,\ldots,x_p)}}{\frac{k\ldb \vv{x} \rdb}{J}} = 0$$ for all $i \geq 1$.
\end{lemma}

\begin{proof}
Recall from Lemma \ref{lem:tonylem} that $g_i \in \lp \vv{x} \rp^2$ for each $p+1 \leq i \leq m$.
We set $J = (\wt{g}_{p+1}, \ldots, \wt{g}_m)$ where $\wt{g_i} = g_i(0,\ldots,0,x_{p+1},\ldots,x_n)$ and immediately have $J \subseteq (x_{p+1}, \ldots, x_n)^2$ and  $I = (x_1,\ldots,x_p) + J$.

As a consequence, for each $1 \leq i \leq p$, we have $x_i$ is a regular element over $k\ldb \vv{x} \rdb / \lp (x_1, \ldots, x_{i-1}) + J \rp$. Thus, $x_1,\ldots,x_p$ is a regular sequence over $k\ldb \vv{x} \rdb/J$ which gives the desired Tor-vanishing.
\end{proof}

\begin{prop}\label{prop:Ires}
Suppose $I = \lp x_1, \ldots, x_p \rp + J$. If $\mathfrak{J}$ is a minimal free resolution of $k\ldb \vv{x} \rdb/J$ over $k \ldb \vv{x} \rdb$, then the minimal free resolution of $k\ldb \vv{x} \rdb/I$ over $k \ldb \vv{x} \rdb$ is given by $K \otimes_{k \ldb \vv{x} \rdb} \mathfrak{J}$ where $K = K^{k \ldb \vv{x} \rdb}\lp x_1, \ldots, x_p \rp$ is the Koszul complex.
\end{prop}

\begin{proof}
Since $x_1, \ldots, x_p$ is a regular sequence over $k\ldb \vv{x} \rdb$, it is resolved by the Koszul complex $K = K^{k \ldb \vv{x} \rdb}\lp x_1, \ldots, x_p \rp$. Since the sequence is also regular over $k\ldb \vv{x} \rdb/J$, we have $\Hm{i}{K \otimes_{k \ldb \vv{x} \rdb} \mathfrak{J}} = 0$ for $i > 0$ and 
\[\Hm{0}{K \otimes_{k \ldb \vv{x} \rdb} \mathfrak{J}} \overset{\ref{lem:torvanVA}}{\cong}  \frac{k\ldb \vv{x}\rdb}{\lp x_1, \ldots, x_p \rp + J} = \frac{k\ldb \vv{x}\rdb}{I}.\]
Thus, $K \otimes_{k \ldb \vv{x} \rdb} \mathfrak{J}$ resolves $k \ldb \vv{x} \rdb/ I$. We note that the resolution is minimal since the tensor product of minimal resolutions is again minimal.
\end{proof}

\begin{cor}\label{cor:bettiX}
Suppose $I = \lp x_1, \ldots, x_p \rp + J$, then the Betti numbers for $I$ over $k \ldb \vv{x} \rdb$ are given by
\[
\betti[k \ldb \vv{x} \rdb]{\ell}{\frac{k\ldb \vv{x} \rdb}{I}} = \sum_{i = 0}^{\ell} \tbinom{p}{\ell - i} \betti[k \ldb \vv{x} \rdb]{i}{\frac{k\ldb \vv{x} \rdb}{J}}.
\]
\end{cor}

\begin{proof}
From Proposition \ref{prop:Ires}, one deduces \[
\betti[k \ldb \vv{x} \rdb]{\ell}{\frac{k\ldb \vv{x} \rdb}{I}} = \sum_{i = 0}^{\ell} \betti[k \ldb \vv{x} \rdb]{\ell - i}{\frac{k\ldb \vv{x} \rdb}{ \lp x_1, \ldots, x_p \rp }} \betti[k \ldb \vv{x} \rdb]{i}{\frac{k\ldb \vv{x} \rdb}{J}}.
\] The result follows by noting that $\betti[k \ldb \vv{x} \rdb]{\ell - i}{k\ldb \vv{x} \rdb / \lp x_1, \ldots, x_p \rp } = \tbinom{p}{\ell - i}$.
\end{proof}

Similarly, Proposition \ref{prop:Ires} can be used to deduce the following.

\begin{cor}\label{cor:poin}
Suppose $I = \lp x_1, \ldots, x_p \rp + J$, then the Poincar\'{e} series of $I$ over $k \ldb \vv{x} \rdb$ is given by $P_{k \ldb \vv{x} \rdb / I}^{k \ldb \vv{x} \rdb} (t) = (1 + t)^p P_{k \ldb \vv{x} \rdb / J}^{k \ldb \vv{x} \rdb} (t).$
\end{cor}

\section{Fiber Products over Residue Fields}\label{Results}

Throughout this section we let $k$ be a field and consider the power series ring $R = k\ldb\vv{x},\vv{y}\rdb$ where $\vv{x}$ and $\vv{y}$ are lists of variables. In particular, we set $\vv{x} = x_1, \ldots, x_n$ and $\vv{y} = y_1, \ldots, y_{n'}$. Moreover, we write $\m$ for the maximal ideal of $R$. The fiber products we consider are of the form $\frac{k\ldb \vv{x} \rdb}{I} \times_k \frac{k\ldb \vv{y} \rdb}{I'}$. To apply the results of \cite{MR4470165}, we observe that
\[\frac{k\ldb \vv{x} \rdb}{I} \times_k \frac{k\ldb \vv{y} \rdb}{I'} \cong \frac{k\ldb \vv{x}, \vv{y} \rdb}{I + \lp \vv{y} \rp} \times_k \frac{k\ldb \vv{x}, \vv{y} \rdb}{\lp \vv{x} \rp + I'} \cong \frac{k \ldb \vv{x}, \vv{y}\rdb}{I + \lp \vv{xy} \rp + I'} \]
where we write $I$ (and $I'$) when working over $k \ldb \vv{x} \rdb$ (and $k \ldb \vv{y} \rdb$) and $k \ldb \vv{x}, \vv{y} \rdb$.

Moving forward, we drop the assumptions that $I \subseteq \lp \vv{x} \rp^2$ and $I' \subseteq \lp \vv{y} \rp^2$. To work around these dropped assumptions, set the following notation.

\begin{notation}\label{notation}
Throughout this section we will consider $I \subseteq \lp \vv{x} \rp$ and $I' \subseteq \lp \vv{y} \rp$. For simplicity, we will write $I$ (and $I'$) regardless if we are working over $k \ldb \vv{x} \rdb$ (and $k \ldb \vv{y} \rdb$) or $R = k \ldb \vv{x}, \vv{y} \rdb$. We fix the following notation.
	\begin{enumerate}[i)]
		\item $p~=~\dim_k\left(\frac{I + \lp \vv{x} \rp^2}{\lp \vv{x} \rp^2}\right)$ and $q~=~\dim_k\left(\frac{I' + \lp \vv{y} \rp^2}{\lp \vv{y} \rp^2}\right)$
		\item Take $I = (x_1, \ldots, x_p) + J$ and $I' = (y_1, \ldots, y_q) + J'$ as in Lemma \ref{lem:torvanVA}.
		\item We set $W = R / \lp \vv{\wt{x}}, \vv{\wt{y}} \rp$ where $\vv{\wt{x}} = x_{p+1}, \ldots, x_n$ and $\vv{\wt{y}} = y_{q+1}, \ldots, y_{n'}$.
		\item Throughout, we set $F = \frac{k \ldb \vv{x} \rdb}{I} \times_k \frac{k \ldb \vv{y} \rdb}{I'}$ and $\wt{F} = \frac{R}{J + \lp \vv{\wt{y}} \rp} \times_W \frac{R}{\lp \vv{\wt{x}} \rp + J'}$.
	\end{enumerate}
\end{notation}

\begin{lemma}\label{lem:defIdeal}
The defining ideal of $F$ over $R$ is given by
 \[J + \lp \vv{\wt{x}} \vv{\wt{y}} \rp + J' + \lp x_1, \ldots, x_p \rp + \lp y_1, \ldots, y_q \rp \]
 where $\vv{\wt{x}} = x_{p + 1}, \ldots, x_n$ and $\vv{\wt{y}} = y_{q + 1}, \ldots, y_{n'}$.
\end{lemma}

\begin{proof}
Using \cite{MR4470165} yields $F \cong  k \ldb \vv{x}, \vv{y} \rdb / \lp I, \vv{x} \vv{y}, I' \rp$. Thus, the defining ideal is $$I + \lp \vv{x} \vv{y} \rp + I' = J + \lp \vv{x} \vv{y} \rp + J' + \lp x_1, \ldots, x_p \rp + \lp y_1, \ldots, y_q \rp.$$ However, we wish to write the defining ideal with a minimal set of generators. In particular, we show that we can remove any generator of the form $x_i y_j$ where $1 \leq i \leq p$ or $1 \leq j \leq q$. This leaves us with $\{x_iy_j : p < i \leq n, q < j \leq n'\}$. We denote the ideal generated by the set as $\lp \vv{\wt{x}} \vv{\wt{y}} \rp$ and observe \[I + \lp \vv{x} \vv{y} \rp + I' = J + \lp \vv{\wt{x}} \vv{\wt{y}} \rp + J' + \lp x_1, \ldots, x_p \rp + \lp y_1, \ldots, y_q \rp. \qedhere\]
\end{proof}

This next proposition uses Lemma \ref{lem:defIdeal} to decompose the fiber product $F$ into the tensor product of a complete intersection ring and a fiber product that can be resolved using techniques from \cite{MR4470165}.

\begin{prop}\label{prop:fibIso}
Under the set-up of Notation \ref{notation}, one has
 \[
 F \cong \wt{F} \otimes_R \frac{R}{\lp x_1, \ldots, x_p, y_1, \ldots, y_{q}\rp}
 \]
\end{prop}

\begin{proof}
To start, we note that Lemma \ref{lem:torvanVA} and \cite{MR4470165}*{Proposition 2.1} combine to produce $\wt{F} \cong R / \lp J + \lp \vv{\wt{x}} \vv{\wt{y}} \rp + J' \rp$. We now observe the following.
\begin{align*}
F &\overset{\ref{lem:defIdeal}}{\cong} \frac{R}{J + \lp \vv{\wt{x}} \vv{\wt{y}} \rp + J' + \lp x_1, \ldots, x_p, y_1, \ldots, y_q \rp} \\
 &\cong \wt{F} \otimes_R \frac{R}{\lp x_1, \ldots, x_p, y_1, \ldots, y_{q}\rp} \qedhere
\end{align*}
\end{proof}

We will use this isomorphism to construct a minimal free resolution of our fiber product. To do this, we first need the following results.

\begin{lemma}\label{lem:regseqFull}
The sequence $ x_1, \ldots, x_p, y_1, \ldots, y_q$ is regular over $\wt{F}$.
\end{lemma}

\begin{proof}
By construction, one has $J + \lp \vv{\wt{x}} \vv{\wt{y}} \rp + J' \subseteq \lp x_{p+1}, \ldots, x_n, y_{q+1}, \ldots, y_{n'} \rp$ with none of the generators containing $x_1, \ldots, x_p$ or $y_1,\ldots,y_q$. Combining this above with the fact $x_1, \ldots, x_p, y_1, \ldots, y_q$ is a list of variables, we conclude that the list of variables is a regular sequence over $\wt{F}$.
\end{proof}

\begin{rem}\label{remark:regular}
In light of Lemma \ref{lem:regseqFull}, we get the minimal free resolution of $F$ over $R$ by tensoring (over $R$) the minimal free resolutions of $R / \lp x_1, \ldots, x_p, y_1, \ldots, y_q \rp$ and $\wt{F}$. Since the former is a quotient by a list of variables, it is resolved by the Koszul complex on the sequence. The latter is resolved by the mapping cone in \cite{MR4470165}*{Theorem 5.2}. These observations provide us with the tools needed to recover the Betti numbers (over $R$) of our fiber product.
\end{rem}

In the following theorem, we make use of two identities for binomial coefficients; Pascal's Identity and Vandermonde's Identity: \[\tbinom{n}{k} + \tbinom{n}{k + 1} = \tbinom{n + 1}{k + 1} \hbox{ and } \tbinom{m + n}{k} = \sum_{i + j = k} \tbinom{m}{i} \tbinom{n}{j}.\]

\begin{theorem}\label{thm:Betti}
Using Notation \ref{notation}, the Betti numbers of $F$ over $R$ are given by $\beta_0^R(F) = 1$ and, for~$t \geq 1$,
\begin{align*}
\beta_t^R(F) =& \tbinom{n + n'}{t + 1} - \tbinom{n' + p + 1}{t + 1} - \tbinom{n + q + 1}{t + 1} + \tbinom{p + q + 1}{t + 1} \\
 & \qquad + \sum_{w + z = t} \lp \beta_w^R \lp \tfrac{R}{I} \rp \tbinom{n'}{z} + \tbinom{n}{z} \beta_w^R\lp \tfrac{R}{I'} \rp\rp.
\end{align*}
\end{theorem}

\begin{proof}
Since $x_1, \ldots, x_p, y_1, \ldots, y_q$ is regular (see Remark \ref{remark:regular}), we know \[\beta_{\ell}^R \lp \tfrac{R}{\lp x_1, \ldots, x_p, y_1, \ldots, y_{q}\rp} \rp = \tbinom{p + q}{\ell}.\] By Lemma \ref{lem:regseqFull}, the sequence $x_1, \ldots, x_p, y_1, \ldots, y_q$ is regular over $\wt{F}$. As a consequence, Proposition \ref{prop:fibIso} produces the formula \[\beta_t^R \lp F \rp = \sum_{\ell + k = t} \beta_{\ell}^R \lp \wt{F} \rp \tbinom{p + q}{k} = \tbinom{p + q}{t} + \sum_{\underset{\ell \neq 0}{\ell + k = t}} \beta_{\ell}^R \lp \wt{F} \rp \tbinom{p + q}{k} .\]
Next, \cite{MR4470165}*{Corollary 3.6, Corollary 5.3} gives us the following expression:
\begin{align*}
\beta_{\ell}^R \lp \wt{F} \rp =& \tbinom{n + n' - p - q}{\ell + 1} - \tbinom{n - p + 1}{\ell + 1} - \tbinom{n' - q + 1}{\ell + 1} \\
 & \qquad + \sum_{i + j = \ell} \lp \beta_i^R \lp \tfrac{R}{J} \rp \tbinom{n' - q}{j} + \tbinom{n - p}{j} \beta_i^R \lp \tfrac{R}{J'} \rp \rp.
\end{align*}
To use this expression in the formula for $\beta_t^R \lp F \rp$, we compute the following.
\begin{align}
\sum_{\underset{\ell \neq 0}{\ell + k = t}} \tbinom{n + n' - p - q}{\ell + 1} \tbinom{p + q}{k}  &= \sum_{\underset{u \neq 0, 1}{u + k = t + 1}} \tbinom{n + n' - p - q}{u} \tbinom{p + q}{k} \nonumber \\
 & \hspace{-1cm} = \sum_{u + k = t + 1} \tbinom{n + n' - p - q}{u} \tbinom{p + q}{k} - \tbinom{p + q}{t + 1} - \lp n + n' - p - q \rp \tbinom{p + q}{t} \nonumber \\
 & \hspace{-1cm} = \tbinom{n + n'}{t + 1} - \tbinom{p + q}{t + 1} - \lp n + n' - p - q \rp \tbinom{p + q}{t} \label{eqn:1} 
\end{align}
The same argument can be used to obtain the following two formulas.
\begin{align}
\sum_{\underset{\ell \neq 0}{\ell + k = t}} \tbinom{n - p + 1}{\ell + 1} \tbinom{p + q}{k} = \tbinom{n + q + 1}{t + 1} - \tbinom{p + q}{t + 1} - \lp n - p + 1 \rp \tbinom{p + q}{t} \label{eqn:2}
\end{align}
\begin{align}
\sum_{\underset{\ell \neq 0}{\ell + k = t}} \tbinom{n'- q + 1}{\ell + 1} \tbinom{p + q}{k} = \tbinom{n' + p + 1}{t + 1} - \tbinom{p + q}{t + 1} - \lp n' - q + 1 \rp \tbinom{p + q}{t} \label{eqn:3}
\end{align}
Next, we compute the following using Corollary \ref{cor:bettiX}.
\begin{align}
\sum_{\underset{\ell \neq 0}{\ell + c = t}} \sum_{i + j = \ell} \beta_i^R \lp \tfrac{R}{J} \rp \tbinom{n' - q}{j} \tbinom{p + q}{c} &= -\tbinom{p + q}{t} + \sum_{i + j + a + b = t} \beta_i^R \lp \tfrac{R}{J} \rp \tbinom{n' - q}{j} \tbinom{p}{a} \tbinom{q}{b} \nonumber \\
 & \hspace{-3cm} = -\tbinom{p + q}{t} + \sum_{w + z = t} \lp \sum_{i + a = w} \beta_i^R \lp \tfrac{R}{J} \rp \tbinom{p}{a} \rp \lp \sum_{j + b = z} \tbinom{n' - q}{j}  \tbinom{q}{b} \rp \nonumber \\
 & \hspace{-3cm} \overset{\ref{cor:bettiX}}{=} -\tbinom{p + q}{t} + \sum_{w + z = t} \beta_w^R \lp \tfrac{R}{I} \rp \tbinom{n'}{z} \label{eqn:4}
\end{align}
Using the same strategy, we obtain the following.
\begin{align}
\sum_{\underset{\ell \neq 0}{\ell + c = t}} \sum_{i + j = \ell} \beta_i^R \lp \tfrac{R}{J'} \rp \tbinom{n - p}{j} \tbinom{p + q}{c} = -\tbinom{p + q}{t} + \sum_{w + z = t} \beta_w^R \lp \tfrac{R}{I'} \rp \tbinom{n}{z} \label{eqn:5}
\end{align}
Using Equations (\ref{eqn:1}) -- (\ref{eqn:5}), we compute the desired formula below.
\begin{align*}
\beta_t^R \lp F \rp =& \tbinom{p + q}{t} + \sum_{\underset{\ell \neq 0}{\ell + k = t}} \beta_{\ell}^R \lp \wt{F} \rp \tbinom{p + q}{k} \\
 =&  \tbinom{p + q}{t} + \tbinom{n + n'}{t + 1} - \tbinom{n' + p + 1}{t + 1} - \tbinom{n + q + 1}{t + 1} + \tbinom{p + q}{t + 1} \\
 & \qquad + \sum_{w + z = t} \beta_w^R \lp \tfrac{R}{I} \rp \tbinom{n'}{z} + \sum_{w + z = t} \beta_w^R \lp \tfrac{R}{I'} \rp \tbinom{n}{z} \\
 =& \tbinom{n + n'}{t + 1} - \tbinom{n' + p + 1}{t + 1} - \tbinom{n + q + 1}{t + 1} + \tbinom{p + q + 1}{t + 1} \\
 & \qquad + \sum_{w + z = t} \lp \beta_w^R \lp \tfrac{R}{I} \rp \tbinom{n'}{z} + \beta_w^R \lp \tfrac{R}{I'} \rp \tbinom{n}{z} \rp \qedhere
\end{align*}
\end{proof}

With this theorem, we can obtain a formula relating the Poincar\'e series for $F$ (as in Theorem \ref{thm:Betti}) to the series for $R/I$ and $R/I'$. However, we will bypass this approach using \cite{MR4470165}*{Corollary 5.4}.

\begin{theorem}\label{thm:Poin}
Using Notation \ref{notation}, the Poincar\'e series of $F$ over $R$ satisfies \[\frac{ P_F^R(t) - \lp 1 + t \rp^{n'} P_{R/I}^R(t) - \lp 1 + t \rp^{n} P_{R/I'}^R(t) + \lp 1 + t \rp^{n + n'} }{\lp \lp 1 + t \rp^{n \vphantom{n'}} - \lp 1 + t \rp^p \rp \lp \lp1 + t \rp^{n'} - \lp 1 + t \rp^q \rp} = \frac{t + 1}{t}.\]
\end{theorem}

\begin{proof}
Proposition \ref{prop:fibIso} and Lemma \ref{lem:regseqFull} give $P_{F}^R(t) = (1 + t)^{p + q} P_{\wt{F}}^R(t)$. With this in mind, we consider \cite{MR4470165}*{Corollary 5.4} applied to $\wt{F}$ to get \[\tfrac{ P_{\wt{F}}^R(t) - \lp 1 + 1 \rp^{n' - q} P_{R/J}^R(t) - \lp 1 + 1 \rp^{n - p} P_{R/J'}^R(t) + \lp 1 + t \rp^{(n - p) + (n' - q)}}{\lp \lp 1 + t \rp^{n \vphantom{n'} - p} - 1 \rp \lp \lp 1 + t \rp^{n' - q} - 1 \rp} = \tfrac{t + 1}{t}.\] To replace $P_{\wt{F}}^R(t)$ with $P_F^R(t)$, we multiply the left-hand side of the expression by $\frac{(1 + t)^{p + q}}{(1 + t)^{p + q}}$. Since Corollary \ref{cor:poin} tells us that $P_{R/I}^R(t) = (1 + t)^p P_{R/J}^R(t)$ and $P_{R/I'}^R(t) = (1 + t)^q P_{R/J'}^R(t)$, we obtain the desired formula.
\end{proof}

\bibliographystyle{plain}

\end{document}